\UseRawInputEncoding
\documentclass[11pt,twoside,reqno]{amsart}
\usepackage{url}
\usepackage{mathrsfs}
\usepackage{latexsym}
\usepackage{amsmath}
\usepackage{amssymb}
\usepackage{amsthm}
\usepackage{amscd}
\usepackage{amsfonts}
\usepackage{epsfig}
\usepackage{multicol}
\usepackage{CJK}
\usepackage{bbm}
\usepackage{amsmath, amsthm, amsfonts, amssymb,mathtools}
\usepackage{mathrsfs}
\usepackage{enumitem} 

\numberwithin{equation}{section}

\def \Re{\textup{Re}}

\theoremstyle{plain}
\newtheorem{thm}{Theorem}[section]
\newtheorem{lem}[thm]{Lemma}
\newtheorem{prop}[thm]{Proposition}
\newtheorem{rem}[thm]{Remark}
\newtheorem{cor}[thm]{Corollary}

\newtheorem{conj}[thm]{Conjecture}

\newcommand\restr[2]{{
		\left.\kern-\nulldelimiterspace 
		#1 
		\right|_{#2} 
}}

\newcommand{\pf}{\noindent\begin {proof}}
\newcommand{\epf}{\end{proof}}
\makeatletter

\begin{document}
		\title[Inequalities for $\zeta(s)-\psi(1-s)$]{Inequalities for $\zeta(s)-\psi(1-s)$ related to a conjecture of Henry}
	
	\author{Liwen Gao}
	\address{Liwen Gao: Department of Mathematics, Nanjing University,
		Nanjing 210093, China;  gaoliwen1206@smail.nju.edu.cn}
	
	\author{Xuejun Guo$^{\ast}$}
	\address{Xuejun Guo$^{\ast}$: Department of Mathematics, Nanjing University,
		Nanjing 210093, China;  guoxj@nju.edu.cn}
	\address{$^{\ast}$Corresponding author}
	
	\thanks{The authors are supported by National Natural Science Foundation of China (Nos. 12231009).}
	
	\date{}
	\maketitle
	
	\noindent
	
	\begin{abstract}
		In this paper we investigate analytic inequalities related to a conjecture of Henry involving the difference between the Riemann zeta function and the digamma function. By treating $\zeta(s)-\psi(1-s)$ as a unified analytic object, we establish its strict convexity and monotonicity on suitable intervals. Moreover, we obtain explicit boundary limits of the derivative, expressed in terms of $\pi$, $\log (2\pi)$ and Stieltjes constants. These results lead to new inequalities for $\zeta(s)-\psi(1-s)$ and shed further light on the conjecture.
	\end{abstract}
	\medskip
	
	
	\vskip 10pt
	
	\section{Introduction}
	The Riemann zeta function is fundamental in analytic number theory.  It is defined for $\Re(s)>1$	by the absolutely convergent Dirichlet series
	$$\zeta(s)=\sum_{n=1}^{\infty} \frac{1}{n^s},$$
	and admits a meromorphic continuation to the whole complex plane with a simple pole at $s=1$. In particular, its Laurent expansion at $s=1$ is given by
	\begin{equation}\label{zetaL}
	\zeta(s)
	= \frac{1}{s-1}
	+ \sum_{k=0}^{\infty} \frac{(-1)^k\,\gamma_k\,(s-1)^k}{k!},
	\end{equation}
	where $\gamma_k$ denote the Stieltjes constants (see Stieltjes's paper \cite{s1} and Ferguson's work \cite{f}). For $k=0$, this reduces to the Euler–Mascheroni constant 
	$\gamma_0 \approx 0.577216$.
	
	Let $\psi(s)=\frac{\Gamma'(s)}{\Gamma(s)}$ be digamma function, where
	$\Gamma$ denotes the usual gamma function. The digamma function is deeply related to the Riemann zeta function. One has
	\begin{equation}\label{digammaL}
	\psi(s+1)= -\gamma_0 + \sum_{j=2}^{\infty} (-1)^j \zeta(j)\, s^{j-1},
	\qquad |s|<1.
\end{equation}

	In \cite{h1}, Henry proposed the following conjecture.
		\begin{conj}\label{conj}
		In this conjecture we define $ b' := \gamma_0+\frac12$ and
		$b :=\gamma_0-\frac12$.
		Then if $0<s<1$, we conjecture that the relation
		\begin{equation}\label{congformula}
			\pi\cot(\pi s)+s < \zeta(s)-\psi(s)
			< \pi\cot(\pi s)+ b'\,s+b
		\end{equation}
		holds.
	\end{conj}
	Motivated by this conjecture,  we investigate strict convexity and monotonicity  of $F(s)$ on $(0,1)$ and determine the precise boundary behavior $F'(s)$, obtaining explicit constants involving $\pi$, $\log(2\pi)$, and the first Stieltjes constant. Our approach relies on Stieltjes expansions of the Riemann zeta function with series representations of the digamma function, and sharp asymptotic estimates, which may be applicable to other combinations of special functions. Our main theorem is as follows.
\begin{thm}\label{Fap}
	Let $F(s)=\zeta(s)-\psi(1-s)$ be a real-valued function defined on the interval $0<s<1$. Then the following assertions hold:
	\begin{enumerate}
		\item $F''(s)>0$ for all $s\in(0,1)$; i.e., $F$ is strictly convex on $(0,1)$.
		\item $F'(s)>0$ for all $s\in(0,1)$; i.e., $F$ is strictly increasing on $(0,1)$.
		\item The boundary behavior of $F'(s)$ is given by
			$$
			\lim_{s\to 0^+} F'(s) = \frac{\pi^2}{6}-\frac{1}{2}\log(2\pi)\quad ,\quad
			\lim_{s\to 1^-} F'(s) = \frac{\pi^2}{6}-\gamma_1.$$

	\end{enumerate}
\end{thm}
	The following corollary, which follows from Theorem \ref{Fap}, implies Conjecture \ref{conj}.
\begin{cor}\label{cor1}
For $0<s<1$, we have the following inequality,
$$	s < F(s) < b'\,s + b $$
 where $b's+b$ is the best possible linear upper bound for $F(s)$ on $(0,1)$.
	\end{cor}
	\begin{rem}
		Figure \ref{C4P3} illustrates the behavior of the functions involved in the conjectured
		inequality
		\[
		s < F(s) < b'\,s + b \quad (0<s<1).
		\]
		The plot serves to provide intuition for the validity of the conjecture
		and to visualize the bounds, while the proof relies entirely on analytic arguments.
		\end{rem}
		
		\begin{figure}\label{C4P3}
			\centering
			\includegraphics[height=8cm]{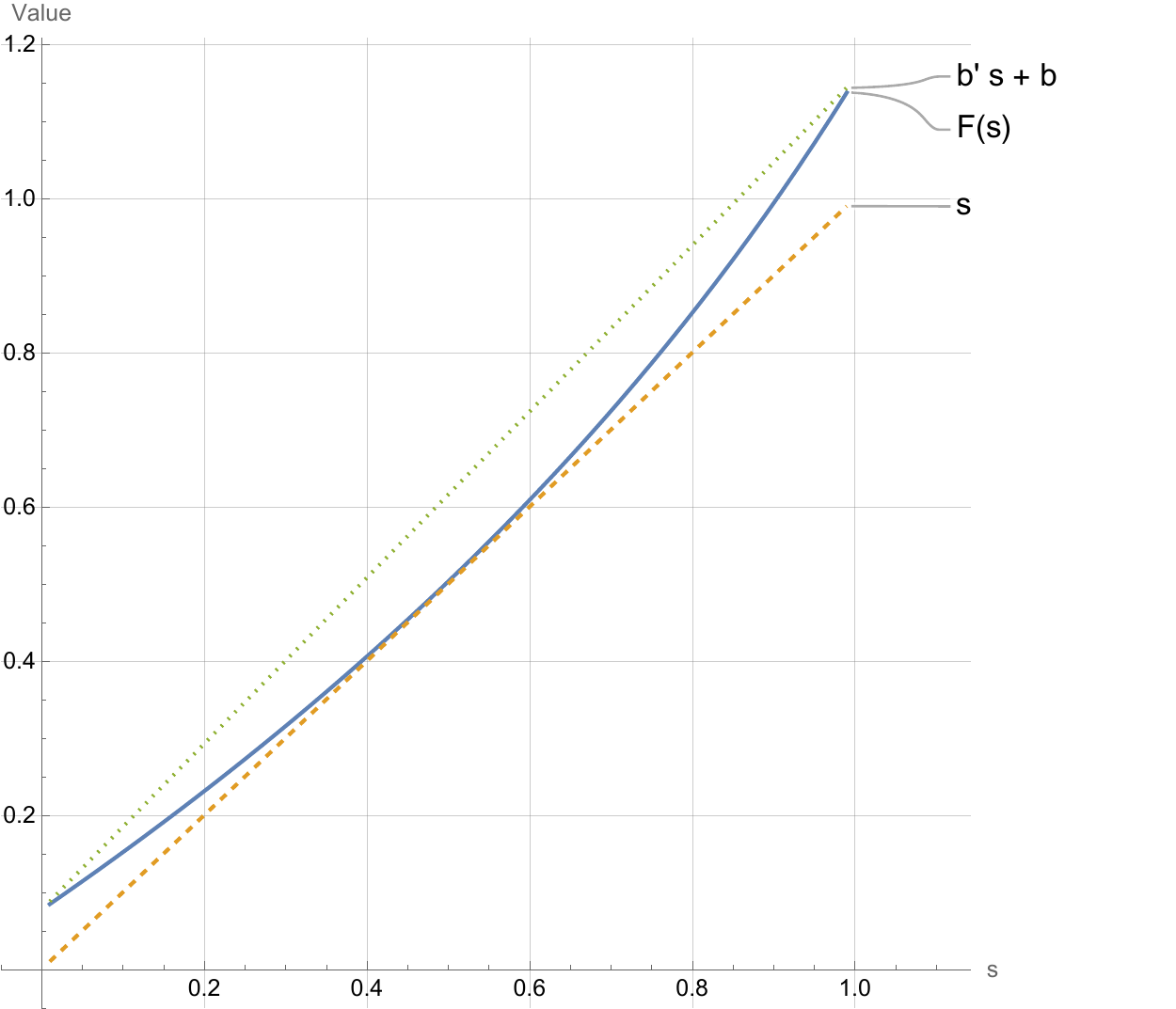}
			\caption{Comparison of $F(s)$ with linear bounds for \\$0<s<1$.}
		\end{figure}
		\begin{rem}
		It has been observed in the literature that inequalities of the type
		considered in this conjecture are related to certain elementary criteria
		for the Riemann hypothesis.
		In particular, assuming the conjecture holds, one may reformulate aspects
		of the Riemann hypothesis in terms of the behavior of $\zeta(x)$ on the
		interval $0<x<1$.
		We do not pursue this direction here and refer the reader to \cite{y} and \cite{ky} for a detailed analysis.
		\end{rem}
	\section{Preliminaries}
Before proving the main theorem, we establish several preliminary results. 

		\begin{lem}[\cite{kt}]\label{digammae} 
		For $\Re (s)>0$, the digamma function admits the integral representation
		\begin{equation*}
			\psi(s)+\gamma_0
			= \int_0^1 \frac{1 - x^{s-1}}{1-x}\,dx.
		\end{equation*}
	\end{lem}
	\begin{lem}[\cite{sc}]\label{Ri}
	When $\Re(s)>-1$ holds, the relation
	\begin{equation*}
		\zeta(s)+\frac{1}{1-s}
		= \frac12 - s \int_1^{\infty} \frac{\{t\}-\frac12}{t^{s+1}}\,dt
	\end{equation*}
	holds, where $\{t\}$ denotes the fractional part of $t$.
	\end{lem}
	\begin{lem}\label{fsup}
		Let $f\in C^{2}([0,1])$. Then
		\[
		\int_0^1 \left(u-\frac12\right) f(u)\,du
		=\frac{(f(1)-f(0))}{8}-\frac{(f'(1)+f'(0))}{48}+ \int_0^1 \frac{\left(u-\frac12\right)^3}{6}f''(u)\,du.
		\]
	\end{lem}
	
	\begin{proof}
	The result follows directly by integration by parts.
	\end{proof}
\begin{lem}\label{converge uniformly}
	The improper integrals
	\[
	\int_1^{\infty} \frac{\{t\}-\frac12}{t^{s+1}}\,dt
	\quad\text{and}\quad
	\int_0^1 \frac{1-t^{-s}}{1-t}\,dt
	\]
	converge uniformly for $s$ on any closed subinterval of $(0,1)$.
	Moreover, for any positive integer $m$, the integrals obtained by differentiating $m$ times with respect to $s$ also converge uniformly on such subintervals.
\end{lem}

		\begin{proof}
			Fix $0<a<b<1$. We prove uniform convergence on $[a,b]$.
				Let
			\[
			I_1(s):=\int_{1}^{\infty}\frac{\{t\}-\frac12}{t^{s+1}}\,dt,
			\qquad
			I_2(s):=\int_{0}^{1}\frac{1-t^{-s}}{1-t}\,dt.
			\]
			\medskip
			\noindent
			\emph{Uniform convergence of $I_1$.}
			Since $|\{t\}-\tfrac12|\le \tfrac12$, for $t\ge 1$ and $s\in[a,b]$,
			\[
			\left|\frac{\{t\}-\frac12}{t^{s+1}}\right|
			\le \frac12\,t^{-a-1},
			\]
			and the majorant $t^{-a-1}$ is integrable on $[1,\infty)$. Hence
			$\int_1^\infty (\{t\}-\frac12)t^{-s-1}\,dt$ converges uniformly for $s\in[a,b]$
			by the Cauchy criterion for improper integrals.
			
			\medskip
			\noindent
			\emph{Uniform convergence of $I_2$.}
			For $t\in(0,1)$ and $s\in[a,b]$,
			\[
			\left|\frac{1-t^{-s}}{1-t}\right|
			\le \frac{1+t^{-s}}{1-t}.
			\]
			Splitting the integral at $t=\tfrac12$, we have
			\[
			\frac{1+t^{-s}}{1-t}
			\ll
			\begin{cases}
				1+t^{-b}, & t\in(0,\tfrac12],\\[4pt]
				1, & t\in[\tfrac12,1),
			\end{cases}
			\]
			where the implied constant depends only on $a,b$.
			Since $b<1$, both majorants are integrable on the corresponding intervals,
			and the uniform convergence of $I_2(s)$ on $[a,b]$ follows.
			
			\medskip
			\noindent
			\emph{Derivatives with respect to $s$.}
			For $m=1,2$, differentiating under the integral sign yields integrands of the form
			\[
			(\log t)^m\frac{\{t\}-\frac12}{t^{s+1}}
			\quad\text{and}\quad
			\frac{(\log t)^m\,t^{-s}}{1-t}.
			\]
			Arguing as above, these are dominated on $[a,b]$ by integrable functions
			$t^{-a-1}(\log t)^m$ on $[1,\infty)$ and $(\log(1/t))^m t^{-b}$ on $(0,1)$,
			respectively. Hence the corresponding improper integrals converge uniformly
			on $[a,b]$.
		\end{proof}
	
			We recall the following lemma from \cite[p.~67]{pp}.
		\begin{lem}\label{Dd}
			Let
			\[
			F(s)=D(f)(s):=\sum_{n=1}^{\infty} \frac{f(n)}{n^{s}}
			\]
			be a Dirichlet series convergent in the half-plane $\Re(s)>\sigma$.
			Then for every integer $k\ge 1$, the series $F(s)$ is $k$ times termwise differentiable
			in the half-plane $\Re(s)>\sigma$, i.e.,
			\[
			F^{(k)}(s)
			=
			(-1)^k \sum_{n=1}^{\infty} \frac{f(n)(\log n)^k}{n^{s}},
			\qquad \Re(s)>\sigma.
			\]
		\end{lem}

		 The Riemann zeta function converges on $\Re(s)>1$. By Lemma \ref{Dd}, we obtain the following proposition.
			\begin{prop}\label{Rmoto}
			If  s is the complex number satisfying $\Re(s)>1$, then
			$$ \sum_{n=1}^{\infty}\frac{\log(n)}{n^s}=-\zeta'(s)\quad,\quad\sum_{n=1}^{\infty}\frac{\log^2(n)}{n^s}=\zeta''(s)\quad,\quad\sum_{n=1}^{\infty}\frac{\log^3(n)}{n^s}=-\zeta'''(s). $$
			Moreover, when $s\in \mathbb{R}$ and $2<s<3$,
			$$\zeta'(s)<0 \quad,\quad \zeta''(s)>0\quad,\quad \zeta'''(s)<0.$$
			\end{prop}

\section{Proof of our Main result}
We now turn to the proof of our main results.

\begin{proof}[Proof of Theorem \ref{Fap}]
	We prove the assertions in the order $(1)\to(3)\to(2)$.
	
	\medskip
	\noindent
	\textbf{Proof of (1).}
Fix $ s \in(0,1)$. Choose $\epsilon_s>0$ sufficiently small such that $\epsilon_s<s<1-\epsilon_s$.
By Lemma \ref{digammae} and \ref{Ri},

$$F(s)=\frac12 - s \int_1^{\infty} \frac{\{t\}-\frac12}{t^{s+1}}\,dt+\frac{1}{s-1}
-\int_0^1 \frac{1 - t^{-s}}{1-t}\,dt+\gamma_0.$$

Since the integrands, together with their partial derivatives with respect to $s$, are continuous and converge uniformly  for $s$ on the interval $\epsilon_s\le s\le 1-\epsilon_s$ by Lemma \ref{converge uniformly}, the following computation is justified.
Differentiate term by term and we obtain
\[
F'(s)=
\int_{1}^{\infty}\frac{(\{t\}-\frac12)(s\log t-1)}{t^{s+1}}\,dt
-\frac1{(s-1)^2}
-\int_{0}^{1}\frac{t^{-s}\log(t) }{1-t}\,dt.
\]
Differentiate again and we obtain,
\[
F''(s)=
\int_{1}^{\infty}\frac{(\{t\}-\frac12)(2-s\log(t) )(\log(t) )}{t^{s+1}}\,dt
+\frac2{(s-1)^3}
+\int_{0}^{1}\frac{t^{-s}\log^2 (t)}{1-t}\,dt.
\]
Let $$J(s):=\int_{1}^{\infty}\frac{(\{t\}-\frac12)(2-s\log t)(\log t)}{t^{s+1}}\,dt$$
and $$P(s):=\frac2{(s-1)^3}
+\int_{0}^{1}\frac{t^{-s}\log^2 t}{1-t}\,dt.$$
To prove the convexity of $F(s)$, it suffices to show that $0<P(s)<\infty $ and $|J(s)|<P(s)$.

In fact, using the geometric series expansion
$$\frac{1}{1-t}=\sum_{n=0}^{\infty}t^n,\qquad |t|<1,$$
we may rewrite \(P(s)\) as follows,
\begin{align*}
	P(s)&=\frac2{(s-1)^3}
	+\int_{0}^{1}\frac{t^{-s}\log^2 (t)}{1-t}\,dt\\
	&=\frac2{(s-1)^3}
	+\sum_{n=0}^{\infty}\int_{0}^{1}\log^2(t) t^{n-s}\,dt\\
	&=\frac2{(s-1)^3}
	+\sum_{n=0}^{\infty}\frac{2}{(n-s+1)^3}\\
	&=\sum_{n=2}^{\infty}\frac{2}{(n-s)^3},\\
	\end{align*}
	where the penultimate step uses the standard integral formula 
	$$\int_{0}^{1}\log^2(t) t^{a}\,dt=\frac{2}{(a+1)^3}.$$
	
	Consequently, for \(0<s<1\),
	$$P(s)=\sum_{n=2}^{\infty}\frac{2}{(n-s)^3}\ge 2\sum_{n=2}^{\infty}\frac{1}{n^3}=
	2(\zeta(3)-1)\approx0.40411380632.$$
	
	Meanwhile,
	$$P(s)=\sum_{n=2}^{\infty}\frac{2}{(n-s)^3}\le\sum_{n=2}^{\infty}\frac{2}{(n-1)^3}=2\zeta(3)<\infty.$$
	To complement this estimate, we now consider the integral term $J(s)$.
	\begin{align*}
J(s):&=\int_{1}^{\infty}\frac{(\{t\}-\frac12)(2-s\log(t))(\log(t))}{t^{s+1}}\,dt\\
&=\sum_{n=1}^{\infty}\int_{n}^{n+1}\frac{(\{t\}-\frac12)(2-s\log (t))(\log t)}{t^{s+1}}\,dt\\
&=\sum_{n=1}^{\infty}\int_{0}^{1}\frac{(t-\frac12)(2-s\log (t+n))(\log (t+n))}{(t+n)^{s+1}}\,dt.\\
\end{align*}
Let $$h_n(t)=\frac{(2-s\log (t+n))(\log (t+n))}{(t+n)^{s+1}}.$$
It is easy to compute that
$$h_n'(t)=\frac{2-(4s+2)\log(t+n)+s(s+1)\log^2(t+n)}{(t+n)^{s+2}},$$
and
$$h_n''(t)=\frac{-6(s+1)+(6s^2+12s+4)\log(t+n)-s(s+1)(s+2)\log^2(t+n)}{(t+n)^{s+3}}.$$

So by Lemma \ref{fsup},
	\begin{align*}
	J(s)&=\frac{1}{8}\sum_{n=1}^{\infty}(h_n(1)-h_n(0))-\frac{1}{48}\sum_{n=1}^{\infty}(h'_n(1)+h'_n(0))\\
	&+\frac{1}{6}\sum_{n=1}^{\infty}\int_{0}^{1}(t-\frac{1}{2})^3h_n''(t)\,dt\\
	&:=I_1+I_2+I_3.\\
\end{align*}
We proceed to estimate the three terms separately.

\medskip
\noindent\textbf{Estimate of \(I_1\).}
Let $$g(x)=\frac{(2-s\log(x))\log(x)}{x^{s+1}}.$$ It is easy to see that $h_n(0)=g(n)$, $h_n(1)=g(n+1)$. Then
\begin{align*}
I_1=\frac{1}{8}\sum_{n=1}^{\infty}(h_n(1)-h_n(0))&=\frac{1}{8}\lim\limits_{N\rightarrow\infty}\sum_{n=1}^{N}(h_n(1)-h_n(0))\\
&=\frac{1}{8}\lim\limits_{N\rightarrow\infty}\sum_{n=1}^{N}(g(n+1)-g(n))=-\frac{1}{8}g(1)=0
\end{align*}

\medskip
\noindent\textbf{Estimate of \(I_2\).}
Let $$g_1(x)=\frac{2-(4s+2)\log x+s(s+1)\log^2 x}{x^{s+2}}.$$ It is easy to see that $h_n'(0)=g_1(n)$, $h_n'(1)=g_1(n+1)$. Then by Proposition \ref{Rmoto},
\begin{align*}
\sum_{n=1}^{\infty}(h'_n(1)+h'_n(0))&=g_1(1)+2\sum_{n=2}^{\infty}g_1(n)\\
&=4\zeta(s+2)-2+(8s+4)\zeta'(s+2)+s(s+1)\zeta''(s+2),
\end{align*}
and the functions $\zeta$  and $\zeta'' $ are strictly decreasing on $(2,3)$ and $\zeta'$ is strictly increasing $(2,3)$. In addition, we have $\zeta>0$ and $\zeta''>0$, whereas $\zeta'<0$ on this interval. Then
\begin{align*}
	&4\zeta(s+2)-2+(8s+4)\zeta'(s+2)+s(s+1)\zeta''(s+2)\\
	&\le 4\zeta(2)-2+4\zeta'(3)+2\zeta''(2)
	\approx7.765791.\\
\end{align*}
Hence
$$|I_2|\le \frac{1}{48}\sum_{n=1}^{\infty}(h'_n(1)+h'_n(0))\le \frac{7.765791}{48}=0.161787$$

\medskip
\noindent\textbf{Estimate of \(I_3\).}
Denote $$\frac{1}{6}\int_{0}^{1}(t-\frac{1}{2})^3h_n''(t)\,dt$$ by $E_n(s)$. Then $I_3=\sum\limits_{n=1}^{\infty}E_n$. We have
	\[
|E_n(s)|
\le \sup_{t\in [0,1]}|h_n''(t)|\int_0^1\left|\frac{(u-\frac12)^3}{6}\right|du
=\frac{1}{192}\,\sup_{t\in [0,1]}|h_n''(t)|,
\]
since $\int_0^1|(u-\frac12)^3|\,du=\frac1{32}$.

One can notice that
\begin{align*}
	&\sum_{n=1}^{\infty}\sup_{t\in [0,1]}|h_n''(t)|\\
	&=\sum_{n=1}^{\infty}\sup_{t\in [0,1]}\left|\frac{-6(s+1)+(6s^2+12s+4)\log(t+n)-s(s+1)(s+2)\log^2(t+n)}{(t+n)^{s+3}}\right|\\
	&\le \sum_{n=1}^{\infty}\frac{6\log^2(1+n)+22\log(1+n)+12}{n^3}.
	\end{align*}
	Hence
	\[
	\sum_{n=1}^{\infty}|E_n(s)|
	\le \frac{1}{192}\sum_{n=1}^{\infty}\frac{6\log^2(n+1)+22\log(n+1)+12}{n^{3}}
	=: \frac{1}{192}\,\Sigma.
	\]
	
	To estimate $\Sigma$, fix $N\ge2$ and split $\Sigma=\Sigma_{\le N}+\Sigma_{>N}$, where
	\[
	\Sigma_{\le N}:=\sum_{n=1}^{N}\frac{6\log^2(n+1)+22\log(n+1)+12}{n^{3}}.
	\]
	For the tail, note that for $x\ge1$ we have $\log(x+1)\le \log(2x)=\log x+\log 2$. Then 
	\begin{align*}
	\Sigma_{>N}&\le \int_{N}^{\infty} \frac{6\log^2(x+1)+22\log(x+1)+12}{x^{3}}\,dx\\
	&\le \int_{N}^{\infty}\frac{6(\log x+\log2)^2+22(\log x+\log2)+12}{x^{3}}\,dx.\\
	\end{align*}
	The latter integral can be evaluated explicitly by integration by parts, giving a closed form
	upper bound of the shape
	\[
	\Sigma_{>N}\le \frac{3(\log N)^2 + (14+6\log (2))\log N + (13+14\log (2)+3\log^2(2)) }{N^{2}}.
	\]

	Taking $N=200$, via Mathematica, a direct computation of $\Sigma_{\le 200}$ together with the above
	explicit tail estimate yields
	\[
	\Sigma<40.697,
	\]
	and therefore
	\[
	I_3\le \sum_{n=1}^{\infty}|E_n(s)|
	\le \frac{\Sigma}{192}<\frac{40.697}{192}\approx0.211964,
	\]
	uniformly for all $0<s<1$.
	
	All in all, 
	\begin{align*}
	|J(s)|&\le|I_1|+|I_2|+|I_3|\\
	&\le 0.161787+0.211964=0.373751<0.40411380632\le P(s).\\
	\end{align*}
	This completes the proof of (1).
	
		\medskip
	\noindent
	\textbf{Proof of (3).}

	We study the behavior of $F'(s)=\zeta'(s)+\psi'(1-s)$ as $s\to 1^-$ and $s\to 0^+$.
		\medskip
	\noindent
	\textbf{Behavior as $s \to 1^-$.}
	First, recall the Laurent expansion of the Riemann zeta function at $s=1$,
	\[
	\zeta(s)=\frac{1}{s-1}
	+ \sum_{k=0}^{\infty} \frac{(-1)^k\,\gamma_k\,(s-1)^k}{k!},
	\]
	which implies
	\[
	\zeta'(s)
	=-\frac{1}{(s-1)^2}
	+ \sum_{k=1}^{\infty} \frac{(-1)^k\,\gamma_k\,(s-1)^{k-1}}{(k-1)!}.
	\]
	In particular, the finite part of $\zeta'(s)$ at $s=1$ equals $-\gamma_1$.
	
	Next, we use the classical power series expansion of the digamma function by
	\cite{gr},
		\begin{equation}\label{digammaderie}
	\psi(z+1)=\psi(z)+\frac{1}{z}
	= -\gamma_0 + \sum_{j=2}^{\infty} (-1)^j \zeta(j)\, z^{j-1},
	\qquad |z|<1.
	\end{equation}
	Differentiating termwise, we obtain
	\[
\psi'(z)
= \frac{1}{z^{2}}
+ \sum_{j=2}^{\infty} (-1)^j \zeta(j)\,
(j-1)\, z^{\,j-2}.
\]
	Replacing $z$ by $1-s$ yields
	\begin{equation*}
\psi'(1-s)
= \frac{1}{(1-s)^{2}}
+ \sum_{j=2}^{\infty} (-1)^{j} \zeta(j)\,
(j-1)\, (1-s)^{\,j-2}.
\end{equation*}
	
	Combining the above expansions, 
	\[
	\lim_{s\to 1^-} F'(s)
	=\lim_{s\to 1^-} (\zeta'(s)+\psi'(1-s))= -\gamma_1 + \zeta(2)
	= \frac{\pi^2}{6}-\gamma_1,
	\]
			\medskip
	\noindent
	\textbf{Behavior as $s \to 0^+$.}
It is well known that $\zeta'(0)=-\frac{1}{2}\log(2\pi)$. By taking $z=-s$ in formula (\ref{digammaderie}), we have
	$$	\psi(1-s)
	= -\gamma_0 - \sum_{j=2}^{\infty}\zeta(j)\, s^{j-1}, \qquad |s|<1.$$
	Differentiating termwise, we obtain
		$$	\frac{d}{ds}\psi(1-s)
	=  - \sum_{j=2}^{\infty}\zeta(j)(j-1)\, s^{j-2}, \qquad |s|<1.$$
	Hence $\lim\limits_{s\to 0^+}\frac{d}{ds}\psi(1-s)=-\zeta(2)$. All in all,
	$$\lim_{s\to 0^+}F'(s)=-\frac{1}{2}\log(2\pi)+\zeta(2)=\frac{\pi^2}{6}-\frac{1}{2}\log(2\pi),$$
	which completes the proof of (3).
	
	\medskip
\noindent
\textbf{Proof of (2).}
It follows immediately from (1) and (3).
\end{proof}
Now we give the proof of Corollary \ref{cor1}.
\begin{proof}[Proof of Corollary \ref{cor1}]
	Since
	$$\pi\cot(\pi s)=\psi(1-s)-\psi(s),$$
	the inequality (\ref{congformula}) is equivalent to
	$$s < \zeta(s)-\psi(1-s)< b'\,s+b.
	$$
	
			A direct computation can obtain	$F(0)=\zeta(0)-\psi(1)=-\frac{1}{2}+\gamma_0$.
Using formula (\ref{zetaL}) and (\ref{digammaL}), it is easy to obtain that  $\lim\limits_{s\to 1^-}F(s)=2\gamma_0$.
	Then by Theorem \ref{Fap} (1),
	$$\zeta(s)-\psi(1-s)=F(s)<F(0)+(F(1)-F(0))s=(\gamma_0-\frac12)+(\gamma_0+\frac12)s.$$
	This establishes the right side of the inequality.
	
	Let $G(s)=F(s)-s$.
	Since 
			\[
	\lim_{s\to 0^+} G'(s)=\frac{\pi^2}{6}-\frac{1}{2}\log(2\pi)-1<0,
	\qquad
	\lim_{s\to 1^-} G'(s)=\frac{\pi^2}{6}-\gamma_1-1>0.
	\]
$G(s)$ has a unique minimum $s_0$ in $(0,1)$ such that $G'(s_0)=0$.
Via Mathematica, $s_0\approx0.484993$ and $G(s_0)\approx0.00306469>0$.
Hence $G(s)>0$ for $s\in (0,1)$ and the left side of the inequality is established. Moreover, 
$$\lim_{s\to 1^-}F(s)=\lim_{s\to 1^-}(b's+b)=2\gamma_0.$$
So $b's+b$ is the best possible linear upper bound for $F(s)$ on $(0,1)$,
which completes the proof of Corollary \ref{cor1}.
	\end{proof}

	\section{Concluding Remarks}
	Hurwitz \cite{h2} introduced the zeta function
	$$\zeta(s,q)=\sum_{n=0}^{\infty}\frac{1}{(n+q)^s},$$
	which is now known as the Hurwitz zeta function. It is one generalization of the Riemann zeta function. Berndt gave the Laurent series of 
 the Hurwitz zeta function in \cite{b}, i.e.,
	$$
	\zeta(s,q)
	= \frac{1}{s-1}
	+ \sum_{k=0}^{\infty}\frac{(-1)^k\,\gamma_k(q)\,(s-1)^k}{k!},
	$$
	where $\gamma_k(q)$ denotes the $k$-th generalized Stieltjes constant.
	
	Although the Hurwitz zeta function shares many analytic properties with the Riemann zeta function, the presence of the additional parameter 
$q$ makes the derivation of explicit inequalities more delicate. It would therefore be natural to ask whether our results can be extended to Hurwitz zeta function $\zeta(s,q)$. It is an interesting question, and we encourage readers to pursue this question.

\subsubsection*{\noindent\textbf{Funding}} {\small The authors are supported by National Natural Science Foundation of China (Nos. 12231009).}
\subsubsection*{\noindent\textbf{Data Availability}} {\small No additional data are available.}

\subsection*{\normalfont\Large\bfseries Declarations}
\par
\subsubsection*{\noindent\textbf{\small Conflict of interest:}} {\small The authors have no relevant financial or non-financial interests to disclose.}

\end{document}